\documentclass{ws-jta}

\def\draftnote{\today}%
\def\trimmarks{}
\def\catchline#1#2#3#4#5{}
\def\received#1{\today\par}
\def\revised#1{}

\usepackage{amsthm}

\usepackage{cleveref}

\DeclareMathOperator{\SL}{SL}
\DeclareMathOperator{\GL}{GL}
\newcommand{\iso}{\cong}
\newcommand{\cover}{\widetilde}

\DeclareMathOperator{\rank}{rank}
\newcommand{\integer}{\mathbb{Z}}
\newcommand{\rational}{\mathbb{Q}}
\newcommand{\real}{\mathbb{R}}
\newcommand{\complex}{\mathbb{C}}
\newcommand{\GG}{\mathbf{G}}
\newcommand{\ZZ}{\mathbf{Z}}
\newcommand{\SSL}{\mathop{\mathbf{SL}}}
\newcommand{\bigtimes}{\mathop{\text{\Large$\times$}}}

\newtheorem{thm}[equation]{Theorem}
\newtheorem{CSP}[equation]{Theorem}
\newtheorem{cor}[equation]{Corollary}
\newtheorem{lem}[equation]{Lemma}

\crefformat{CSP}{the Congruence Subgroup Property (#2#1#3)}

\theoremstyle{definition}
\newtheorem{rem}[equation]{Remark}
\newtheorem{rems}[equation]{Remarks}
\newtheorem{notation}[equation]{Notation}
\newtheorem*{ack}{Acknowledgments}

\crefformat{rems}{Remark~#2#1#3}

\makeatletter
\def\@citestyle{\m@th\upshape\mdseries}
\let\citeform\@firstofone
\def\@cite#1#2{{%
  \@citestyle[\citeform{#1}\if@tempswa, #2\fi]}}
\makeatother

\newcommand{\fullcref}[2]{\cref{#1}(\ref{#1-#2})}

\begin{document}

\markboth{Dave Witte Morris and Robert J.\ Zimmer}{Volume-preserving actions of simple algebraic $\rational$-groups on low-dimensional manifolds}

\catchline{}{}{}{}{}

\title{Volume-preserving actions of simple algebraic $\rational$-groups on low-dimensional manifolds}

\author{Dave Witte Morris}

\address
{Department of Mathematics and Computer Science, University of Lethbridge, 
\\ Lethbridge, Alberta, T1K~3M4, Canada
\\Dave.Morris@uleth.ca, http://people.uleth.ca/$\sim$dave.morris/}

\author{Robert J.\ Zimmer}

\address{Office of the President, University of Chicago, 
\\ Chicago, Illinois 60637, USA
\\president@uchicago.edu, http://president.uchicago.edu/}

\maketitle

\begin{history}
\received{(Day Month Year)}
\revised{(Day Month Year)}
\end{history}

\begin{abstract}
We prove that $\SL(n,\rational)$ has no nontrivial, $C^\infty$, volume-preserving action on any compact manifold of dimension strictly less than~$n$. More generally, suppose $\GG$ is a connected, isotropic, almost-simple algebraic group over~$\rational$, such that the simple factors of every localization of $\GG$ have rank $\ge 2$. If there does not exist a nontrivial homomorphism from $\GG(\real)^\circ$ to $\GL(d,\complex)$,
then every $C^\infty$, volume-preserving action of $\GG(\rational)$ on any compact $d$-dimensional manifold must factor through a finite group. 
\end{abstract}



\section{Introduction}

The second author has conjectured that if $\GG$ is a simple algebraic $\rational$-group, and $\rank_{\real}\GG \ge 2$, then every $C^\infty$, volume-preserving action of the arithmetic group $\GG(\integer)$ on an compact manifold of small dimension must be \emph{finite}. (This means that the action factors through the action of a finite group. See \cite{Fisher-AroundZimmer} for a precise statement of the conjecture and a survey of progress on this problem.) In this paper, we show that known results imply the analogue of the conjecture with $\GG(\rational)$ in the place of $\GG(\integer)$. For example, we establish:

\begin{thm} \label{SLnQNoAct}
$\SL(n,\rational)$ has no nontrivial, $C^\infty$, volume-preserving action on any compact manifold of dimension strictly less than~$n$.
\end{thm}

\begin{rem}
$\SL(n,\rational)$ contains large finite subgroups whenever $n$~is large (such as an elementary abelian group of order $2^{n-1}$). Therefore, topological arguments imply that if $M$ is any compact manifold, then there is some~$n$, such that $\SL(n,\rational)$ has no nontrivial, $C^0$ action on~$M$ (see \cite[Thm.~2.5]{MannSu-ActionsOfPGrps}). However, unlike in \cref{SLnQNoAct}, the value of~$n$ depends on details of the topology of~$M$, not just its dimension, because every finite group acts freely on some compact, connected, 2-dimensional manifold \cite[Thm.~7.12]{FarbMargalit-MappingClassGrps}.
\end{rem}

The nontrivial part of \cref{SLnQNoAct} (namely, when $n \ge 3$) is a special case of the following much more general result:

\begin{thm} \label{GQNOACT}
Assume:
	\begin{enumerate} \renewcommand{\theenumi}{\alph{enumi}}
	\item $\GG$ is an isotropic, almost-simple, linear algebraic group over\/~$\rational$, such that, for every place~$v$ of\/~$\rational$, the\/ $\rational_v$-rank of every simple factor of\/ $\GG(\rational_v)$ is at least two,
	\item \label{GQNOACT-d}
	$d \in \integer^+$, such that there are no nontrivial, continuous homomorphisms from\/ $\GG(\real)^\circ$ to\/ $\GL(d,\complex)$,
	and
	\item $G$ is a subgroup of finite index in\/~$\GG(\rational)$.
	\end{enumerate}
Then every $C^\infty$, volume-preserving action of~$G$ on any $d$-dimensional compact manifold~$M$ is finite.
\end{thm}

\begin{rem}[anisotropic groups] \label{anisotropic}
Assume, for simplicity, that $\GG$ is connected. Then the assumption that $\GG$ is isotropic can be eliminated if we add two hypotheses on the universal cover~$\cover\GG$: 
	\begin{romanlist}
	\item $\cover\GG(\rational)$ is projectively simple,
	and
	\item sufficiently large $S$-arithmetic subgroups of~$\cover\GG$ have the Congruence Subgroup Property.
	\end{romanlist}
Both of these hypotheses are known to be true unless $\GG$ is anisotropic of type $A_n$, $D_4$, or~$E_6$. See \cref{AnisotropicPf} for more details.
\end{rem}

\begin{rems}
	\begin{enumerate}
	\item To satisfy the requirement that the $\rational_v$-rank of every simple factor of $\GG(\rational_v)$ is at least two, it suffices to let $\GG$ be an absolutely almost-simple algebraic group over~$\rational$, such that $\rank_{\rational} \GG \ge 2$.  In particular, we can take $\GG = \SSL_n$ with $n \ge 3$. This yields \cref{SLnQNoAct}.
	\item The assumption that the subgroup~$G$ has finite index can be replaced with the weaker assumption that it contains the commutator subgroup $[\GG^\circ(\rational), \GG^\circ(\rational)]$.
	\item Our bound on the dimension~$d$ of~$M$ is probably not sharp. 
	In particular, we conjecture that $\SL(n,\rational)$ has no volume-preserving action on any compact manifold of dimension strictly less than $n^2 - 1$.
	In the general case, it should suffice to assume that $\GG(\real)^\circ$ has no simple factor of dimension $\le d$. 	 
	\end{enumerate}
\end{rems}

\section{Proof of \cref{GQNOACT}}

Assume the situation of \cref{GQNOACT}. By passing to a subgroup of finite index, we assume that $\GG$ is connected.

\begin{notation}
	\begin{enumerate}
	\item $\cover\GG$ is the universal cover of~$\GG$. (We may realize $\cover\GG$ as a Zariski-closed subgroup of $\SSL_N$, for some~$N$ \cite[Thm.~8.6, p.~63]{Humphreys-AlgicGrps}, so $\cover\GG(R)$ is defined for any integral domain~$R$ of characteristic zero.)
	\item $\pi \colon \cover\GG \to \GG$ is the natural homomorphism.
	\item $\ZZ$ is the kernel of~$\pi$ (so $\ZZ$ is a finite, central $\rational$-subgroup of~$\cover\GG$).
	\item If $S$ is any finite set of prime numbers:
		\begin{enumerate}
		\item $\integer_S$ is the ring of $S$-integers. That is, $\integer_S = \integer[1/p_1,\ldots,1/p_r]$, where $S = \{p_1,\ldots,p_r\}$.
		\item $\Gamma\!_S = \cover\GG \bigl( \integer_S \bigr)$, so $\Gamma\!_S$ is an $S$-arithmetic subgroup of~$\cover\GG$.
		\item $\widehat\Gamma\!_S$ is the profinite completion of~$\Gamma\!_S$.
		\end{enumerate}
	\item $\integer_p$ is the ring of $p$-adic integers, for any prime~$p$.
	\end{enumerate}
\end{notation}

We begin by recalling a few well-known facts about $\cover\GG(\rational)$:

\begin{lem} \label{G/Univ} 
	\begin{enumerate}
	\item  \label{G/Univ-simple}
	Every proper, normal subgroup of\/ $\cover\GG(\rational)$ is contained in the center of\/~$\cover\GG$, and is therefore finite. 
	\item \label{G/Univ-inG}
	$\pi \bigl( \cover\GG(\rational) \bigr) \subseteq G$.
	\item  \label{G/Univ-abel}
	$\GG(\rational) / \pi \bigl( \cover\GG(\rational) \bigr)$ is an abelian group whose exponent divides\/ $|\ZZ(\complex)|$. {\upshape(}In particular, $\pi \bigl( \cover\GG(\rational) \bigr)$ is a normal subgroup of\/ $\GG(\rational)$.{\upshape)}
	\end{enumerate}
\end{lem}

\begin{proof}
(\ref{G/Univ-simple}) See \cite[Thm.~8.1]{Gille-KneserTits}. 
This relies on our assumption that $\GG$ is isotropic.

(\ref{G/Univ-inG}) Since $G$ has finite index in $\GG(\rational)$, it contains a finite-index subgroup of $\pi \bigl( \cover\GG(\rational) \bigr)$. However, we know from (\ref{G/Univ-simple}) that $\cover\GG(\rational)$ has no proper subgroups of finite index. Therefore $G$ must contain all of $\pi \bigl( \cover\GG(\rational) \bigr)$.

(\ref{G/Univ-abel}) We have the following long exact sequence of Galois cohomology groups \cite[(1.11), p.~22]{PlatonovRapinchukBook}:
	$$ H^0 \bigl( \rational; \cover\GG \bigr) \longrightarrow 
	H^0 \bigl( \rational; \GG \bigr) \longrightarrow
	H^1 \bigl( \rational; \ZZ \bigr) .$$
In other words,
	$$ \cover\GG(\rational) \longrightarrow 
	\GG(\rational) \stackrel{\delta}{\longrightarrow}
	H^1 \bigl( \mathop{\mathrm{Gal}}(\overline{\rational}/ \rational), \ZZ(\complex) \bigr) .$$
Since $\ZZ$ is central in~$\GG$, it is easy to see that the connecting map $\delta$ is a group homomorphism. Therefore, the desired conclusion follows from the observation that multiplication by $|\ZZ(\complex)|$ annihilates the abelian group $H^1 \bigl( \,{*}\, ; \ZZ(\complex) \bigr)$.
\end{proof}

Since $\Gamma\!_S \subset \cover\GG(\rational)$, and $G$ acts on~$M$, \fullcref{G/Univ}{inG} provides an action of~$\Gamma\!_S$ on~$M$ (for any~$S$). The following theorem about this action requires our assumption that there are no nontrivial, continuous homomorphisms from $\GG(\real)^\circ$ to $\GL(\dim M,\complex)$. It also uses our assumption that simple factors of $\GG(\rational_v)$ have rank at least two. (This implies that $\GG(\rational_v)$ has Kazhdan's property $(T)$.)

\begin{thm}[{}{\cite[Cor.~1.3]{Zimmer-SpectrumEntropy}}]
If $S$ is any finite set of prime numbers, then there exist
	\begin{itemlist}
	\item a continuous action of a compact group~$K_S$ on a compact metric space~$X_S$,
	and
	\item a homomorphism $\varphi_S \colon \Gamma\!_S \to K_S$,
	\end{itemlist}
such that the resulting action of\/~$\Gamma\!_S$ on~$X_S$ is measurably isomorphic {\upshape(}a.e.{\upshape)}\ to the action of\/~$\Gamma\!_S$ on~$M$.
\end{thm}

We may assume that $\varphi_S(\Gamma\!_S)$ is dense in~$K_S$. This implies:

\begin{lem}[{}{cf.\ \cite[Cor.~1.5]{Zimmer-SpectrumEntropy}}]
$K_S$ is profinite.
\end{lem}

\begin{proof}
It is an easy consequence of the Peter-Weyl Theorem that every compact group is a projective limit of compact Lie groups \cite[Cor.~2.43, p.~51]{HofmannMorris-CpctGrps}. 
However, since $\GG(\real)$ has no compact factors, the Margulis Superrigidity Theorem \cite[Thm.~B(iii), pp.~258--259]{MargulisBook} tells us that any homomorphism from~$\Gamma\!_S$ into a compact Lie group must have finite image. Since $\varphi_S(\Gamma\!_S)$ is dense in~$K_S$, this implies that $K_S$ is a projective limit of finite groups, as desired.
\end{proof}

Therefore, we may assume $K_S$ is the profinite completion~$\widehat\Gamma\!_S$ of~$\Gamma\!_S$. We have the following well-known description of~$\widehat\Gamma\!_S$ (because $\GG$ is isotropic).

\begin{CSP}[{}Congruence Subgroup Property 
	{\cite{PrasadRaghunathan-CSPMetaplectic,Raghunathan-CSP,Raghunathan-CSP2}}] \label{CSP}
If $S$ is nonempty, then the natural inclusion\/ $\Gamma\!_S \hookrightarrow \bigtimes_{p \notin S} \cover\GG(\integer_p)$ extends to an isomorphism\/ $\widehat\Gamma\!_S \iso \bigtimes_{p \notin S} \cover\GG(\integer_p)$. 
\end{CSP}

Fix a prime number $q \neq 2$. The inclusion $\Gamma\!_{\{2\}} \subset \Gamma\!_{\{2,q\}}$ provides us with an action of~$\Gamma\!_{\{2\}}$ on~$X_{\{2,q\}}$, but this must be isomorphic to the action of~$\Gamma\!_{\{2\}}$ on~$X_{\{2\}}$ (since both are isomorphic to the action on~$M$). Therefore, the action of~$\widehat\Gamma\!_{\{2\}}$ on~$X_{\{2\}}$ must factor through~$\widehat\Gamma\!_{\{2,q\}}$ (a.e.).
Furthermore, if we use \cref{CSP} to identify $\widehat\Gamma\!_S$ with $\bigtimes_{p \notin S} \cover\GG(\integer_p)$, then it is obvious that $\cover\GG(\integer_q)$ is in the kernel of the homomorphism $\widehat\Gamma\!_{\{2\}} \to \widehat\Gamma\!_{\{2,q\}}$.
Therefore, $\cover\GG(\integer_q)$ acts trivially on~$X_{\{2\}}$ (a.e.). 

Since the subgroups $\cover\GG(\integer_q)$ generate a dense subgroup of $\bigtimes_{p\neq2} \cover\GG(\integer_p) \iso  \widehat\Gamma\!_{\{2\}}$, we conclude that $\widehat\Gamma\!_{\{2\}}$ acts trivially (a.e.). Therefore, $\Gamma\!_{\{2\}}$ acts trivially on~$M$ (not just a.e., because $\Gamma\!_{\{2\}}$ acts continuously on~$M$), so the action of $\cover\GG(\rational)$ has an infinite kernel. Hence, \fullcref{G/Univ}{simple} implies that the kernel is all of $\cover\GG(\rational)$.
This means that $\cover\GG(\rational)$ acts trivially on~$M$. 

So the action of~$G$ factors through $G/ \pi \bigl( \cover\GG(\rational) \bigr)$.
From \fullcref{G/Univ}{abel}, we know that this quotient is an abelian group of finite exponent, so the corollary of the following theorem tells us that the action is finite.

\begin{thm}[{}{\cite[Thm.~2.5]{MannSu-ActionsOfPGrps}}] \label{ZpInftyNotOnMfld}
If $A$ is any abelian group of prime exponent, then every $C^0$ action of~$A$ on any compact manifold is finite.
\end{thm}

\begin{cor}
If $A$ is any abelian group of finite exponent, then every $C^0$ action of~$A$ on any compact manifold is finite.
\end{cor}

\begin{proof}
We can assume the exponent of~$A$ is a power of a prime~$p$ (because $A$ is the direct product of its finitely many Sylow subgroups). We can also assume that the action of~$A$ is faithful, so the theorem tells us that $A$ has only finitely many elements of order~$p$. This means the kernel of the homomorphism $x \mapsto x^p$ is finite, so it is easy to prove by induction that $A$ has only finitely many elements of any order~$p^k$. Since $A$ has finite exponent, this implies that $A$ is finite.
\end{proof}

This completes the proof of \cref{GQNOACT}. We now discuss the generalization described in \cref{anisotropic}.

\begin{rem} \label{AnisotropicPf}
The assumption that $\GG$ is isotropic was used in only two places: the projective simplicity of $\cover\GG(\rational)$ (\fullcref{G/Univ}{simple}) and \cref{CSP}. 

The projective simplicity is known to be true unless $\GG$ is anisotropic of type $A_n$, ${}^{3,6}\!D_4$, or~$E_6$ \cite[pp.~513--515]{PlatonovRapinchukBook}. (Projective simplicity obviously fails if there is a nonarchimedean place~$v$, such that $\GG(\rational_v)$ has a compact factor \cite[pp.~510--511]{PlatonovRapinchukBook}. However, compact nonarchimedean factors cannot arise unless $\GG$ is of type~$A_n$ \cite[Thm.~6.5, p.~285]{PlatonovRapinchukBook}. In any case, we have ruled out compact factors by requiring the simple factors of $\GG(\rational_v)$ to have rank at least~$2$.) When there are no compact factors, projective simplicity is also known to be true for inner forms of type ${}^1\!A_n$ \cite[p.~180]{Rapinchuk-CSP}.

For the Congruence Subgroup Property, it suffices to assume that every prime number~$q$ is contained in a finite set~$S$ of prime numbers, such that the congruence kernel $C(S, \cover\GG)$ is central. (This condition is known to be true unless $\GG$ is anisotropic of type $A_n$, $D_4$, or~$E_6$ \cite[Thms.~9.23 and 9.24, pp.~568--569]{PlatonovRapinchukBook}. In fact, for our purposes, it would suffice to know that $C(S, \cover\GG)$ is abelian.) To see that this assumption suffices, note that, for any finite set~$S$ of prime numbers, Strong Approximation \cite[Thm.~7.12, p.~427]{PlatonovRapinchukBook} tells us $\widehat\Gamma\!_S/ C(S,\cover\GG) \iso  \bigtimes_{p \notin S} \cover\GG(\integer_p)$. In particular, $\widehat\Gamma\!_{\emptyset}/ C(\emptyset,\cover\GG) \iso  \bigtimes_{p} \cover\GG(\integer_p)$, so, for each prime~$q$, we may let $\widehat\GG(\integer_q)$ be the inverse image of $\cover\GG(\integer_q)$ in~$\widehat\Gamma\!_{\emptyset}$. The homomorphism $\widehat\Gamma\!_{\emptyset} \to \widehat\Gamma\!_S$ must map $\widehat\GG(\integer_q)$ into $C(S,\cover\GG)$ for all $q \in S$. If $C(S,\cover\GG)$ is abelian, this implies that the image of the commutator subgroup $[\widehat\GG(\integer_q), \widehat\GG(\integer_q)]$ is trivial. Since this is true for all~$q$, we conclude that $[\Gamma\!_{\emptyset}, \Gamma\!_{\emptyset}]$ acts trivially on~$M$. This is sufficient to show that $\cover\GG(\rational)$ acts trivially.
\end{rem}

\begin{ack}
D.\,W.\,M.\ would like to thank the mathematics departments of Indiana University and the University of Chicago for their excellent hospitality while this paper was being written, and would especially like to thank David Fisher for very helpful conversations about this material. 
\end{ack}


\begin{thebibliography}{99}

\bibitem{FarbMargalit-MappingClassGrps}
B.~Farb and D.~Margalit,
\emph{A Primer on Mapping Class Groups,}
(Princeton, 2012).
MR2850125,
ISBN~978-0-691-14794-9.

\bibitem{Fisher-AroundZimmer}
D.~Fisher,
Groups acting on manifolds: around the Zimmer program,
in \emph{Geometry, Rigidity, and Group Actions},
eds. B.~Farb and D.~Fisher,
(Univ. Chicago Press, 2011), pp.~72--157.
MR2807830, ISBN~978-0-226-23788-6.

\bibitem{Gille-KneserTits}
P.~Gille,
Le probl\`eme de Kneser-Tits,
\emph{Ast\'erisque} \textbf{326} (2009) 
39--81. 
MR2605318. 

\bibitem{HofmannMorris-CpctGrps}
K.~H.~Hofmann and S.~A.~Morris,
\emph{The Structure of Compact Groups, 2nd ed.,}
(de Gruyter, 2006).
MR2261490, 
ISBN~978-3-11-019006-9.

\bibitem{Humphreys-AlgicGrps}
J.~E.~Humphreys,
\emph{Linear Algebraic Groups,}
(Springer, 1975).
MR0396773, 
ISBN~0-387-90108-6.

\bibitem{MannSu-ActionsOfPGrps}
L.~N.~Mann and J.~C.~Su,
Actions of elementary $p$-groups on manifolds,
\emph{Trans. Amer. Math. Soc.} \textbf{106} (1963) 115--126.
MR0143840. 

\bibitem{MargulisBook}
G.~A.~Margulis,
\emph{Discrete Subgroups of Semisimple Lie Groups,}
(Springer, 1991).
MR1090825, 
ISBN~3-540-12179-X.

\bibitem{PlatonovRapinchukBook}
V.~Platonov and A.~Rapinchuk,
\emph{Algebraic Groups and Number Theory,}
(Academic, 1994).
MR1278263, 
ISBN~0-12-558180-7.

\bibitem{PrasadRaghunathan-CSPMetaplectic}
G.~Prasad and M.~S.~Raghunathan,
On the congruence subgroup problem: determination of the ``metaplectic kernel'',
\emph{Invent. Math.} \textbf{71} (1983) 21--42.
MR0688260. 

\bibitem{Raghunathan-CSP}
M.~S.~Raghunathan,
On the congruence subgroup problem,
\emph{Inst. Hautes \'Etudes Sci. Publ. Math.} \textbf{46} (1976) 107--161.
MR0507030. 

\bibitem{Raghunathan-CSP2}
M.~S.~Raghunathan,
On the congruence subgroup problem~II,
\emph{Invent. Math.} \textbf{85} (1986) 73--117. 
MR0842049. 

\bibitem{Rapinchuk-CSP}
A.~S.~Rapinchuk,
The congruence subgroup problem,
in \emph{Algebra, K-theory, Groups, and Education (New York, 1997)}, 
eds. T.~Y.~Lam and A.~R.~Magid,
(Amer. Math. Soc., 
1999), pp.~175--188.
MR1732047, ISBN~0-8218-1087-1.

\bibitem{Zimmer-SpectrumEntropy}
R.~J.~Zimmer,
Spectrum, entropy, and geometric structures for smooth actions of Kazhdan groups,
\emph{Israel J. Math.} \textbf{75} (1991) 65--80.
MR1147291. 

\end{thebibliography}
\end{document}